\documentclass[10pt,a4paper,reqno]{amsart}
\usepackage{amsfonts,amsthm,latexsym,amsmath,amssymb,amscd,amsmath, epsf}

\newtheorem{theorem}{Theorem}

\newtheorem{proposition}{Proposition}

\newtheorem{example} {Example}

\newtheorem{conjecture}{Conjecture}

\newcommand {\bC} {\mathbb {C}}
\newcommand {\bR} {\mathbb R}

\newcommand {\bCP} {\mathbb {CP}}
\newcommand {\bP} {\mathbb {P}}
\newcommand {\bRP} {\mathbb {RP}}

\newcommand {\HH} {\mathcal H}

\newcommand {\cP} {\mathcal P}

\newcommand{\LL}{\mathcal L}

\newcommand{\uu}{\mathbf u}
\newcommand{\xx}{\mathbf x}

\theoremstyle{plain}

 \theoremstyle{definition}

\begin{document}
             \numberwithin{equation}{section}

             \title [Canonical forms and moment-generating functions of plane polypols]
             {Canonical forms and moment-generating functions of plane polypols}


\author[B.~Shapiro]{Boris Shapiro}
\address{Department of Mathematics, Stockholm University, SE-106 91, Stockholm,
            Sweden}
\email{shapiro@math.su.se}

\date{\today}

\begin{abstract}
We study two closely related objects associated with plane domains bounded by
rational algebraic arcs: canonical forms in the sense of positive geometry and
normalized moment-generating functions, or Fantappi\`e transforms.  For polygons
these objects are related by polarity: the normalized Fantappi\`e transform of a
polygon is the canonical form of the polar polygon.  For genuinely curved
polypols the same dual-geometric mechanism survives, but the transform is no
longer a rational logarithmic canonical form; rather, it is a holonomic,
generally branched period whose singularities are controlled by vertex
hyperplanes and by the projective dual curves of the nonlinear boundary
components.  We give explicit examples, including sectors and half-disks, and
explain how harmonic moment generating functions arise as one-dimensional
restrictions of the same Fantappi\`e transform.
\end{abstract}

\maketitle


\section{Introduction}
\label{sec:int}

 Given a real algebraic projective hypersurface $\HH\subset \bRP^N$ (not necessarily irreducible), we call a connected component $\Theta$ of the complement $\bRP^N\backslash \HH$ an \emph{algebraic body}. Let $\widetilde \Theta\subset \bR^{N+1}$ be the cone over $\Theta$. By the Laplace transform of $\Theta$ we denote the function
 \begin{equation}\label{eq:Laplace}
\LL_\Theta(U)=\frac{1}{N!}\int_{X\in \widetilde \Theta}e^{-U\cdot X}d^{N+1}X.
 \end{equation}
 For this integral to converge we need the body to be contained in some affine chart in $\bRP^N$ and $U$ to lie in the interior of the dual cone $\widetilde \Theta^{\vee}$. 
 
 Let us provide the affine version which we will use below for bounded algebraic bodies lying in an affine chart of $\bRP^N$.  Writing $X=\rho \widehat X$ and using the standard projective volume convention gives
\[
 d^{N+1}X=\rho^N d\rho\,\langle \widehat X d^N \widehat X \rangle,
\]
where $\langle \widehat X d^N \widehat X \rangle$ denotes the induced projective volume form.  Hence integration in the radial variable gives
\[
\LL_\Theta(U)=\int_{\widehat X\in \Theta}
\frac{\langle \widehat X d^N \widehat X \rangle}
{(\widehat X\cdot U)^{N+1}},
\]
up to the harmless global normalization fixed in \eqref{eq:Laplace}.

In the affine version, we have a bounded algebraic body $\Theta\subset \bR^N$ equipped with coordinates $x_1,x_2,\dots, x_N$.
We define its moments $m_I(\Theta)$ as 
$$m_I(\Theta):=\int_\Theta x_1^{i_1}x_2^{i_2}\dots x_N^{i_N} dx_1dx_2\dots dx_N,$$
where $I=(i_1,i_2,\dots, i_N)$ is an arbitrary multiindex. 
Define the normalized moment generating function
$$F_\Theta(\uu):=\sum_{I:=(i_1,\dots, i_N)\ge 0}\frac{(| I |+N)!}{i_1!\dots i_N!} m_I(\Theta)\uu^I.
$$
It turns out that $F_\Theta(\uu)$ admits an integral presentation as
\begin{equation}\label{eq:main}
F_\Theta(\uu)=N! \int_\Theta \frac {dx_1dx_2\dots dx_N}{(1-\langle \xx, \uu\rangle)^{N+1}}
\end{equation}
which is a special case of the so-called Fantappi\`e transform. 

Up to a constant, $\LL_\Theta(U)$ is the homogenization of $F_\Theta(\uu)$.  For a general algebraic body the analytic continuation of $F_\Theta(\uu)$ is transcendental, and its explicit calculation is a difficult period problem for the affine hypersurface defining $\Theta$.  However, the special case in which $\Theta$ is bounded by rational hypersurfaces is already of interest and is related to the recent activity around positive geometries; compare \cite{AHBL}. It is also related to earlier advances in applied mathematics and computer graphics involving generalized barycentric subdivision; see \cite{Wa}. 

We recall the notion of a canonical form for an algebraic body with rational boundary, following \cite{AHBL}.

Let $X$ be a complex projective algebraic variety, which is the solution set in $\bCP^N$ of a finite set of homogeneous polynomial equations. We will assume that the polynomials have real coefficients. We then denote by $X(\bR)$ the real part of $X$, which is the solution set in $\bRP^N$ of the same set of equations.

A semialgebraic set in $\bRP^N$ is a finite union of subsets, each of which is cut out by finitely many homogeneous real polynomial equations $\{x \in \bRP^N | p(x) = 0\}$ and homogeneous real polynomial inequalities 
$\{x \in \bRP^N | q(x) > 0\}$. To make sense of the inequality $q(x) > 0$, we first find solutions in $\bR^{N +1}\backslash\{0\}$, and then take the image of the solution set in $\bRP^N$.
We define a D-dimensional positive geometry to be a pair $(X,X_{\ge 0})$, where $X$ is an irreducible complex projective variety of complex dimension $D$ and $X_{\ge 0} \subset X(\bR)$ is a nonempty oriented closed semialgebraic set of real dimension $D$ satisfying some technical assumptions 
together with the following recursive axioms:
\begin{itemize}
\item For $D=0$: $X$ is a single point and we must have $X_{\ge 0} =X$. We define the $0$-form $\Omega(X,X_{\ge 0})$ on $X$ to be $\pm 1$ depending on the orientation of $X_{\ge 0}$.
\item For $D>0$: we must have

(P1) Every boundary component $(C,C_{\ge 0})$ of $(X,X_{\ge 0})$ is a positive geometry of dimension $D-1$.

(P2) There exists a unique nonzero rational $D$-form $\Omega(X,X_{\ge 0})$ on $X$ constrained by the residue relation $Res_C \Omega(X, X_{\ge 0}) = \Omega(C, C_{\ge 0})$ along every boundary component $C$, and no singularities elsewhere.
\end{itemize}

The residue operator $Res$ is defined in the following way. Let $\omega$ be a meromorphic form on $X$. Suppose $C$ is an irreducible subvariety of $X$ and $z$ is a holomorphic coordinate whose zero set $z = 0$ locally parametrizes $C$. Let us denote the other holomorphic coordinates collectively as $u$. Then a simple pole of $\omega$ at $C$ is a singularity of the form
$$\omega(u, z) = \omega^\prime(u) \wedge \frac{dz}{z}+ \dots $$
where the $\dots$ denotes terms smooth in the small $z$ limit, and $\omega^\prime(u)$ is a non-zero meromorphic form defined locally on the boundary component. We (locally) define
$$Res_C \omega := \omega^\prime.$$
 If there is no such simple pole, then we define the residue to be zero.
 
 Canonical forms have the following two features which will be used below.  First, they are birationally functorial: if a positive geometry is parametrized by simpler pieces, its canonical form is obtained as a signed push-forward of the logarithmic form on these pieces.  Secondly, for a compact oriented region in an affine plane with a normal-crossing boundary, the canonical form is the unique meromorphic two-form whose only poles are simple poles along the complexified boundary and whose iterated residues at oriented vertices are equal to $\pm 1$.  These two properties are the main reason why planar polypols are natural test examples for positive geometry beyond the polyhedral case.
 
 \medskip
 Our main purpose is to clarify how far the familiar polytope relation between
positive geometries and moment-generating functions extends beyond the
polyhedral case.  The answer is twofold.  First, regular rational polypols admit
canonical forms of Wachspress type: the denominator is the reduced boundary
equation and the numerator is the adjoint which cancels spurious residual
intersections of the complexified boundary.  Secondly, their normalized
moment-generating functions are boundary periods obtained from the same rational
arcs by Green's formula.  For polygons these periods collapse to rational
functions and reproduce the canonical forms of polar polygons.  For curved
polypols they are still governed by projective duality, but they may acquire
algebraic, logarithmic, or arctangent terms and should be regarded as branched
or period-valued analogues of dual canonical forms.

 \medskip
 The paper is organized as follows.  Section~\ref{sec:prel} recalls the
polytope case, introduces polypols, and states the two basic results used
throughout the paper: the Wachspress form of the canonical form and the
boundary-period formula for the Fantappi\`e transform.  Section~\ref{sec:proofs}
contains the proofs, including the important dual statement.  The next two
sections give explicit canonical forms and explicit normalized moment-generating
functions.  Section~\ref{sec:harm} explains the relation with harmonic moments.
The final section summarizes the resulting positive-geometry interpretation and
lists some further questions.

\section{Preliminaries}\label{sec:prel} 

\subsection{Relation between a positive geometry and moment-generating function for polytopes} 
For a convex polytope $P\subset \bR^N$, the corresponding positive geometry is the pair
$$
(\bCP^N,P),
$$
where $P$ is viewed as the closure of an oriented semialgebraic subset of $\bRP^N$ in an affine chart.  Its boundary strata are the faces of $P$, with their induced orientations.  If $P$ is given as the intersection of half-spaces $L_i(x)\geq 0$, then its canonical form has the shape
$$
\Omega_P=\frac{A_P(x)}{\prod_i L_i(x)}dx_1\wedge\cdots\wedge dx_N,
$$
where $A_P$ is the adjoint numerator cancelling the extraneous poles produced by non-face intersections of the supporting hyperplanes.  This is one of the basic examples of positive geometry in \cite{AHBL}; the relation with Wachspress coordinates and adjoints is developed in \cite{Wa,Wa75,Warren,KoRa,GaetzAdjoints}, and the rational moment-transform formula for polytopes is closely related to \cite{GrPaShSh}.

The normalized moment generating function is naturally a dual object.  If
$$
F_P(u)=N!\int_P\frac{dx_1\cdots dx_N}{(1-\langle x,u\rangle)^{N+1}},
$$
then, for a polytope, $F_P(u)du_1\wedge\cdots\wedge du_N$ is the canonical form of the polar positive geometry
$$
P^\circ=\{u\in(\bR^N)^*\mid 1-\langle v,u\rangle\geq 0\quad\hbox{for all vertices }v\hbox{ of }P\},
$$
up to the usual orientation convention.  Thus the apparent equality between canonical forms and moment generating functions should be interpreted after passing to the polar geometry.

\begin{proposition}[Polytope case]\label{prop:Poly}
Let $P\subset \bR^N$ be a full-dimensional polytope with vertices $V(P)$, and assume that the origin is in the interior of $P$ so that the polar body is bounded.  Then
$$
F_P(u)=\frac{A_{P^\circ}(u)}{\prod_{v\in V(P)}(1-\langle v,u\rangle)},
$$
where $A_{P^\circ}$ is the adjoint numerator of the polar polytope.  Equivalently,
$F_P(u)du_1\wedge\cdots\wedge du_N$ is the canonical form of $P^\circ$, up to the normalization fixed by orientation.
\end{proposition}

\subsection{Polypols} 
Following Wachspress, we call a bounded path-connected plane semialgebraic domain whose boundary is a union of real rational arcs a \emph{plane polypol}.  If all boundary components have degree at most two we call it a \emph{polycon}.  We shall always assume that the boundary is regular in the sense that the chosen arcs meet only at finitely many vertices and that the complexified boundary has no component contained in another one.  For such domains, the recent work on adjoints and canonical forms of polypols gives the following positive-geometry statement.

\begin{theorem}[Canonical form of a regular rational polypol]\label{th:CanForm}
Let $\cP\subset \bR^2$ be a regular rational polypol.  Let
$$
B_{\cP}(x,y)=\prod_\ell p_\ell(x,y)
$$
be the reduced equation of its complexified boundary.  Then $(\bCP^2,\overline{\cP})$ is a two-dimensional positive geometry and its canonical form is
$$
\Omega_{\cP}=\kappa_{\cP}\frac{A_{\cP}(x,y)}{B_{\cP}(x,y)}dx\wedge dy,
$$
where $A_{\cP}$ is the Wachspress adjoint curve of $\cP$.  The constant $\kappa_{\cP}$ is determined by the convention that the residue along each oriented boundary arc is the canonical one-form of that arc, and the iterated residues at vertices are $\pm1$.
\end{theorem}

In this form the theorem is essentially the planar case of the Wachspress--positive-geometry correspondence for polypols; see \cite{KohnRanestadSinnWinterPolypols} and the earlier perspective of Wachspress \cite{Wa,Wa75}.  The adjoint condition says geometrically that $A_{\cP}$ passes through the residual intersection points of the complexified boundary curves, i.e. through those intersections which are forced by B\'ezout but are not real vertices of the polypol.  This is exactly the cancellation of spurious poles required by the axioms of a positive geometry.

\begin{theorem}[Fantappi\`e transform of a rational polypol]\label{th:GenFucnt}
Let $\cP\subset\bR^2$ be a regular rational polypol and set
$$
F_{\cP}(u,v)=2\int_{\cP}\frac{dxdy}{(1-ux-vy)^3}.
$$
Then $F_{\cP}$ is obtained by summing explicit one-dimensional integrals over rational boundary arcs.  Consequently it is an elementary multivalued function of $(u,v)$, built from algebraic functions, rational functions, logarithms and arctangent terms; equivalently it is a period of the rational differential form above and is holonomic in the parameters.  If $\cP$ is a polygon, all boundary integrals are rational and Proposition~\ref{prop:Poly} identifies $F_{\cP}$ with the canonical form of the polar polygon.  For genuinely curved boundary arcs one should not expect rationality; the upper half-disk example below already contains algebraic and arctangent terms, while the full disk gives the especially simple smooth-boundary model $2\pi(1-u^2-v^2)^{-3/2}$.
\end{theorem}

Theorem~\ref{th:GenFucnt} shows that the moment transform remains highly
structured for rational polypols, but it also shows why the polygonal statement
cannot be copied verbatim.  The useful precise formulation is the following.

\begin{theorem}[Dual transform statement]\label{th:dual-transform}
Let $\cP\subset\bR^2$ be a regular rational plane polypol and let
\[
        F_{\cP}(u,v)=2\int_{\cP}\frac{dxdy}{(1-ux-vy)^3}.
\]
Then $F_{\cP}$ extends, by analytic continuation, to a holonomic multivalued
function on the complement of an algebraic curve in the dual affine plane.  Its
singularities can occur only when the moving line
\[
        1-ux-vy=0
\]
passes through a vertex of $\cP$ or is tangent to one of the complexified
nonlinear boundary components of $\partial\cP$.

If $\cP$ is a polygon, no nonlinear tangency contribution occurs and
\[
        F_{\cP}(u,v)\,du\wedge dv
\]
is the ordinary canonical form of the polar polygon, up to the global orientation
sign.  If $\cP$ has curved rational boundary arcs, the same form should instead
be regarded as a branched, or twisted, dual canonical object: a period-valued
analogue of a canonical form attached to the moving-line family.  In general it
is not a meromorphic logarithmic canonical form of an ordinary positive geometry.
\end{theorem}

The distinction is essential.  In the polygonal case the singularities are
hyperplanes dual to the vertices and the transform is rational.  In the curved
case new components appear: the projective duals of the boundary curves.  Their
local contribution is often branched rather than logarithmic.  The full disk, although not used below as a canonical-form example, is
the simplest smooth rational-boundary model: its transform is
\[
        F_D(u,v)=\frac{2\pi}{(1-u^2-v^2)^{3/2}},
\]
so the dual conic $1-u^2-v^2=0$ carries a half-integral singularity rather than
a simple logarithmic pole.

\section {Proofs}\label{sec:proofs}
 
\begin{proof}[Proof of Theorem~\ref{th:GenFucnt}]
Put
\[
        Q_{u,v}(x,y)=1-ux-vy.
\]
The formula is first proved for $(u,v)$ in a sufficiently small neighbourhood of the origin.  In this region $Q_{u,v}$ has no zero on $\overline{\cP}$, and the integral defining $F_{\cP}$ is an ordinary convergent integral depending holomorphically on $(u,v)$.  The final statement then follows on the universal cover of the complement of the singular locus by analytic continuation.

Assume first that $u\ne 0$.  A direct calculation gives
\[
   d\left(\frac{dy}{2uQ_{u,v}(x,y)^2}\right)
       =\frac{dx\wedge dy}{Q_{u,v}(x,y)^3}.
\]
Indeed,
\[
 d\left(Q_{u,v}^{-2}dy\right)
 =\frac{2u\,dx\wedge dy}{Q_{u,v}^{3}}.
\]
Stokes' theorem therefore gives
\[
  \int_{\cP}\frac{dxdy}{Q_{u,v}(x,y)^3}
   =\frac{1}{2u}\int_{\partial\cP}\frac{dy}{Q_{u,v}(x,y)^2}.
\]
Equivalently, if $v\ne0$, using
\[
   d\left(-\frac{dx}{2vQ_{u,v}(x,y)^2}\right)
       =\frac{dx\wedge dy}{Q_{u,v}(x,y)^3},
\]
one obtains
\[
  \int_{\cP}\frac{dxdy}{Q_{u,v}(x,y)^3}
   =-\frac{1}{2v}\int_{\partial\cP}\frac{dx}{Q_{u,v}(x,y)^2}.
\]
Multiplying by $2$, according to the normalization in the definition of $F_{\cP}$, gives the boundary representations
\begin{equation}\label{eq:boundary-F-proof}
 F_{\cP}(u,v)=\frac{1}{u}\int_{\partial\cP}\frac{dy}{Q_{u,v}(x,y)^2}
              =-\frac{1}{v}\int_{\partial\cP}\frac{dx}{Q_{u,v}(x,y)^2},
\end{equation}
where the first identity is used on $u\ne0$ and the second on $v\ne0$.  Since both sides are holomorphic near the origin after taking the original area integral as the definition, these formulas determine the same analytic germ and extend across $u=0$ or $v=0$ by continuation.

Now decompose the oriented boundary into finitely many rational arcs
\[
        \partial\cP=\gamma_1\cup\cdots\cup\gamma_s.
\]
For each arc choose a rational parametrization
\[
        \gamma_j:\tau\in [a_j,b_j]\longmapsto (x_j(\tau),y_j(\tau)),
        \qquad x_j(\tau),y_j(\tau)\in\bR(\tau),
\]
compatible with the boundary orientation.  Substituting this parametrization in the first formula in \eqref{eq:boundary-F-proof} gives
\begin{equation}\label{eq:rational-arc-integrals}
 F_{\cP}(u,v)=\frac{1}{u}
       \sum_{j=1}^s\int_{a_j}^{b_j}
       \frac{y_j'(\tau)\,d\tau}
       {\bigl(1-u x_j(\tau)-v y_j(\tau)\bigr)^2},
\end{equation}
whenever $u\ne0$.  Similarly, for $v\ne0$,
\begin{equation}\label{eq:rational-arc-integrals-v}
 F_{\cP}(u,v)=-\frac{1}{v}
       \sum_{j=1}^s\int_{a_j}^{b_j}
       \frac{x_j'(\tau)\,d\tau}
       {\bigl(1-u x_j(\tau)-v y_j(\tau)\bigr)^2}.
\end{equation}
The integrands in \eqref{eq:rational-arc-integrals} and \eqref{eq:rational-arc-integrals-v} are rational functions of $\tau$ whose coefficients are rational functions of $(u,v)$.  Hence each boundary contribution is an explicitly computable integral of a rational differential on $\bP^1$.

By partial fractions over the algebraic closure of the coefficient field $\bC(u,v)$, an antiderivative of such a rational differential is a sum of a rational function of $\tau$ and logarithms of algebraic linear factors in $\tau$.  Over the real coefficient field, conjugate pairs of logarithms may equivalently be written as logarithmic and arctangent terms, with algebraic functions of $(u,v)$ appearing through the roots and discriminants of the denominator.  Evaluation at the endpoints $a_j,b_j$ therefore expresses $F_{\cP}(u,v)$, after analytic continuation, as an elementary multivalued function built from algebraic functions, rational functions, logarithms and arctangent terms.  This proves the asserted boundary-integral representation and the elementary character of the transform.

The holonomicity statement follows from the same representation.  Each summand in \eqref{eq:rational-arc-integrals} is a definite integral, over a fixed interval in the parameter $\tau$, of a rational function in the variables $(\tau,u,v)$.  Such parameter-dependent rational periods are holonomic functions of the parameters; equivalently, differentiating with respect to $u$ and $v$ produces a finite-dimensional module modulo exact rational differentials in $\tau$, which gives linear differential equations with polynomial coefficients annihilating the integral.  Summing over the finitely many arcs preserves holonomicity.

If $\cP$ is a polygon, each arc is a line segment.  On a segment $(x,y)=p+t(q-p)$, $0\le t\le1$, the denominator $Q_{u,v}(x,y)$ is affine in $t$, and the boundary integrals above are rational functions of $(u,v)$.  Summing over the sides gives a rational function.  Proposition~\ref{prop:Poly} identifies precisely this rational function, multiplied by $du\wedge dv$, with the canonical form of the polar polygon, up to the orientation convention fixed there.

Finally, curved rational arcs generally produce non-rational contributions because the rational denominator in the boundary parameter may have quadratic or higher degree factors.  As a smooth rational-boundary model, the full disk gives
\[
        F_{\mathrm{disk}}(u,v)=\frac{2\pi}{(1-u^2-v^2)^{3/2}},
\]
which is algebraic but not rational.  This shows that rationality is special to the polygonal case and completes the proof.
\end{proof}

\begin{proof}[Proof of Theorem~\ref{th:dual-transform}]
The holonomicity and multivalued continuation have already been established in
the proof of Theorem~\ref{th:GenFucnt}.  We recall the argument in the form
needed here.  Write
\[
        Q_{u,v}(x,y)=1-ux-vy
\]
and decompose the oriented boundary into rational arcs
$\gamma_j(\tau)=(x_j(\tau),y_j(\tau))$.  For $u\ne0$ one has
\[
 F_{\cP}(u,v)=\frac{1}{u}
       \sum_j\int_{a_j}^{b_j}
       \frac{y_j'(\tau)\,d\tau}
       {\bigl(1-u x_j(\tau)-v y_j(\tau)\bigr)^2},
\]
and there is an analogous formula with $x_j'$ and $v$ when $v\ne0$.  Thus each
summand is a period of a rational differential in the parameter $\tau$ with
coefficients in $\bC(u,v)$.

The possible singularities of such a period are the usual endpoint and critical
value singularities of a rational one-dimensional integral.  Endpoint
singularities occur when
\[
        1-u x_j(a_j)-v y_j(a_j)=0
        \quad\hbox{or}\quad
        1-u x_j(b_j)-v y_j(b_j)=0.
\]
These are precisely the hyperplanes dual to the vertices of the polypol.  The
remaining singularities occur when the denominator has a multiple zero in the
interior of the complexified parameter curve, that is, when for some boundary
component
\[
        1-u x_j(\tau)-v y_j(\tau)=0,
        \qquad
        u x_j'(\tau)+v y_j'(\tau)=0.
\]
Geometrically this means that the line $1-ux-vy=0$ is tangent to the
complexified boundary curve containing the arc.  Eliminating $\tau$ gives the
projective dual curve of that boundary component.  Hence the singular locus is
contained in the union of the vertex hyperplanes and these dual curves.

For a polygon every boundary component is a line segment.  The moving line has
no nonlinear tangency condition along a side, and the preceding boundary
integrals are rational.  Proposition~\ref{prop:Poly} identifies
$F_{\cP}(u,v)du\wedge dv$ with the canonical form of the polar polygon, up to
orientation.

It remains only to explain why the curved case is not an ordinary logarithmic
canonical form in general.  The full disk gives the simplest local model; for it one computes
\[
        F_D(u,v)\,du\wedge dv=
        \frac{2\pi\,du\wedge dv}{(1-u^2-v^2)^{3/2}}.
\]
At a generic point of the dual conic $q(u,v)=1-u^2-v^2=0$ this has local
exponent $-3/2$.  An ordinary canonical form has, by definition, a simple
logarithmic pole along each boundary component.  Even after a finite branched
cover $q=t^e$ which makes the form single-valued, the pull-back has pole order
$3e/2$ when $e$ is even, and remains multivalued when $e$ is odd.  Thus no such
cover turns the disk transform into a logarithmic canonical form.  This proves
the theorem.
\end{proof}

\section{Examples of canonical forms of polypols}

The purpose of this section is to record a few completely explicit canonical
forms.  They also fix the normalization used in Theorem~\ref{th:CanForm}.
Throughout, the sign of a displayed form depends on the orientation convention;
we choose the orientation for which the iterated residues at positively oriented
vertices are equal to $+1$.  We deliberately do not include the full disk among
the canonical-form examples: with no vertices, its boundary circle is not a
one-dimensional positive geometry in the recursive sense used here.  The disk
will reappear only as a simple model for the moment transform.

\begin{example}[The standard triangle]
Let
\[
        \Delta=\{(x,y):x\geq0,\; y\geq0,\; 1-x-y\geq0\}.
\]
The boundary equation is
\[
        B_\Delta(x,y)=xy(1-x-y).
\]
Since the boundary consists of three lines, the adjoint has degree zero.  Thus
\[
        \Omega_\Delta=\frac{dx\wedge dy}{xy(1-x-y)} .
\]
Indeed, near the vertex $(0,0)$ the last factor is a unit and the iterated
residue of $dx\wedge dy/(xy)$ is $1$.  The same local computation at the other
two vertices gives the remaining signs dictated by the boundary orientation.
Thus the constant numerator is the Wachspress adjoint in this case.
\end{example}

\begin{example}[The square]
For the square
\[
        Q=[-1,1]^2
\]
one has
\[
        B_Q(x,y)=(1-x^2)(1-y^2).
\]
Again the adjoint is constant, and the canonical form is
\[
        \Omega_Q=\frac{4\,dx\wedge dy}{(1-x^2)(1-y^2)} .
\]
For instance, restricting first to the side $y=-1$ gives the residue
$2dx/(1-x^2)$, which is the canonical one-form of the interval $[-1,1]$.
Taking the second residue at $x=1$ gives $1$, up to the standard orientation
sign.  This is the product form for the product positive geometry
$[-1,1]\times[-1,1]$.
\end{example}

\begin{example}[The upper half-disk]
Let
\[
        D_+=\{(x,y):x^2+y^2\leq1,\; y\geq0\}.
\]
Its boundary consists of the interval $[-1,1]$ on the $x$-axis and the upper
semicircle.  The reduced boundary equation is
\[
        B_{D_+}(x,y)=y(1-x^2-y^2).
\]
The total degree is three, hence the adjoint numerator has degree zero.  The
normalization is fixed by requiring that the residue along the diameter be the
canonical form of the interval $[-1,1]$.  This gives
\[
        \Omega_{D_+}=\frac{2\,dx\wedge dy}{y(1-x^2-y^2)} .
\]
Indeed,
\[
        \operatorname*{Res}_{y=0}\Omega_{D_+}
        =\frac{2\,dx}{1-x^2},
\]
which has residues $+1$ and $-1$ at the two endpoints, according to the induced
orientation.  The residue along the circular component is the corresponding
canonical one-form on the arc, with simple poles at the two endpoints.  Thus the
same constant adjoint simultaneously normalizes both boundary components.
\end{example}

\section{Examples of normalized moment-generating functions for polypols}

We keep the normalization used throughout the paper, namely
\[
        F_\Theta(u,v)=2\int_\Theta\frac{dxdy}{(1-ux-vy)^3}.
\]
Thus $F_\Theta(0,0)=2\operatorname{Area}(\Theta)$.  This convention is a useful check on all constants below.

{\bf 1. The disk.}  For the unit disk $D\subset\bR^2$ one has
\[
        F_D(u,v)=\frac{2\pi}{(1-u^2-v^2)^{3/2}} .
\]
Indeed, by rotational invariance it is enough to put $v=0$.  Then
\[
2\int_0^{2\pi}\int_0^1\frac{r\,dr\,d\phi}{(1-ur\cos\phi)^3}
       =2\pi(1-u^2)^{-3/2},
\]
and the stated formula follows by replacing $u^2$ by $u^2+v^2$.  Notice that
$F_D(0,0)=2\pi$, as required by the normalization.  The function is algebraic but not rational.

{\bf 2. The upper half-disk.}  Let
\[
        D_+=\{(r\cos\phi,r\sin\phi):0\le r\le1,\\ 0\le\phi\le\pi\}.
\]
Put
\[
        \rho^2=1-u^2-v^2.
\]
Then
\[
\boxed{
        F_{D_+}(u,v)=
        \frac{\pi+2\arctan\!\left(\frac{v}{\sqrt{1-u^2-v^2}}\right)}{(1-u^2-v^2)^{3/2}}
        +\frac{2v}{(1-u^2)(1-u^2-v^2)} .}
\]
This is the real branch near $(u,v)=(0,0)$, where the square root is positive and the arctangent vanishes at $v=0$.  In particular
$F_{D_+}(0,0)=\pi=2\operatorname{Area}(D_+)$.

For completeness let us indicate the calculation.  After integrating in $r$ one gets
\[
 F_{D_+}(u,v)=\int_0^\pi\frac{d\phi}{(1-u\cos\phi-v\sin\phi)^2}.
\]
With $t=\tan(\phi/2)$ and $s=(1+u)t-v$, this becomes
\[
F_{D_+}(u,v)=\int_{-v}^{\infty}
  \frac{2\big((1+u)^2+(s+v)^2\big)}{(1+u)(s^2+\rho^2)^2}\,ds .
\]
The antiderivative is
\[
\frac{2\arctan(s/\rho)}{\rho^3}
+\frac{1}{1+u}\frac{2s(u+u^2+v^2)/\rho^2-2v}{s^2+\rho^2},
\]
which gives the displayed formula after evaluating at $s=\infty$ and $s=-v$.
With the above branch convention the expression is real near the origin.

{\bf 3. A circular sector.}  Let
\[
\Phi=\{(r\cos\phi,r\sin\phi):0\le r\le1,
       \phi_0\le\phi\le\phi_1\}.
\]
Then the correctly normalized formula is
\[
\begin{aligned}
F_\Phi(u,v)
&=2\int_0^1\int_{\phi_0}^{\phi_1}
       \frac{r\,d\phi\,dr}{(1-r u\cos\phi-rv\sin\phi)^3}  \\
&=\int_{\phi_0}^{\phi_1}
       \frac{d\phi}{(1-u\cos\phi-v\sin\phi)^2} .
\end{aligned}
\]
Thus, setting $\tau_j=\tan(\phi_j/2)$, $j=0,1$, and again
$\rho^2=1-u^2-v^2$, one obtains
\[
F_\Phi(u,v)=
\int_{(1+u)\tau_0-v}^{(1+u)\tau_1-v}
  \frac{2\big((1+u)^2+(s+v)^2\big)}{(1+u)(s^2+\rho^2)^2}\,ds .
\]
Equivalently,
\[
F_\Phi(u,v)=\mathcal A\big((1+u)\tau_1-v\big)
             -\mathcal A\big((1+u)\tau_0-v\big),
\]
where
\[
\mathcal A(s)=
\frac{2\arctan(s/\rho)}{\rho^3}
+\frac{1}{1+u}\frac{2s(u+u^2+v^2)/\rho^2-2v}{s^2+\rho^2} .
\]
This formula also passes the constant check:
\[
        F_\Phi(0,0)=\phi_1-
        \phi_0=2\operatorname{Area}(\Phi).
\]

{\bf 4. The standard triangle.}  Let
\[
        \Delta=\{(x,y):x\geq0,\; y\geq0,\; x+y\leq1\}.
\]
Then
\[
        F_\Delta(u,v)=\frac{1}{(1-u)(1-v)} .
\]
Indeed,
\[
        F_\Delta(u,v)=2\int_0^1\int_0^{1-x}
        \frac{dy\,dx}{(1-ux-vy)^3},
\]
and elementary integration gives the displayed rational function.  Its value at
$(0,0)$ is $1$, which equals $2\operatorname{Area}(\Delta)$.
This is the simplest two-dimensional illustration of Proposition~\ref{prop:Poly}.

{\bf 5. A rectangle.}  Let
\[
        R_{a,b}=[0,a]\times[0,b].
\]
For $u,v\ne0$, with extension by continuity to $u=0$ or $v=0$, one has
\[
\boxed{
        F_{R_{a,b}}(u,v)=
        \frac{1}{uv}\left(
        \frac{1}{1-au-bv}-\frac{1}{1-au}
        -\frac{1}{1-bv}+1\right).}
\]
The formula follows by integrating first in $x$ and then in $y$:
\[
2\int_0^b\int_0^a\frac{dx\,dy}{(1-ux-vy)^3}
 =\frac{1}{u}\int_0^b\left(
 \frac{1}{(1-au-vy)^2}-\frac{1}{(1-vy)^2}\right)dy.
\]
At the origin the removable singularity has value $2ab$, again equal to
$2\operatorname{Area}(R_{a,b})$.

\section{Generating harmonic moments} \label{sec:harm}

The holomorphic, or harmonic, moments of a bounded plane domain $\Theta$ are

a natural one-dimensional shadow of the two-variable moment theory discussed
above.  Put
\[
        \mu_j(\Theta)=\int_\Theta z^j\,dz\,d\bar z,\qquad j=0,1,2,\ldots .
\]
Their ordinary generating function and the weighted generating function are
\[
        S_\Theta(t)=\sum_{j\geq 0}\mu_j(\Theta)t^j,
        \qquad
        G_\Theta(t)=\sum_{j\geq 0}(j+1)\mu_j(\Theta)t^j.
\]
Thus
\[
        G_\Theta(t)=(tS_\Theta(t))'.
\]
Equivalently,
\[
        S_\Theta(t)=\int_\Theta\frac{dz\,d\bar z}{1-tz},
        \qquad
        G_\Theta(t)=\int_\Theta\frac{dz\,d\bar z}{(1-tz)^2}.
\]
For $t\ne0$ one has
\[
        d\left(\frac{1}{t}\frac{d\bar z}{1-tz}\right)
        =\frac{dz\wedge d\bar z}{(1-tz)^2},
\]
and hence, with the usual limiting interpretation at $t=0$,
\begin{equation}\label{eq:harmonic-boundary-G}
        G_\Theta(t)=\frac{1}{t}\int_{\partial\Theta}
        \frac{d\bar z}{1-tz}.
\end{equation}
Similarly,
\[
        d\left(-\frac{\bar z\,dz}{1-tz}\right)
        =\frac{dz\wedge d\bar z}{1-tz},
\]
so that
\begin{equation}\label{eq:harmonic-boundary-S}
        S_\Theta(t)=-\int_{\partial\Theta}\frac{\bar z\,dz}{1-tz}.
\end{equation}
These formulae show that, for rational polypols, $S_\Theta$ and $G_\Theta$
are obtained from the same boundary calculus as $F_\Theta(u,v)$, but after
restricting the moving affine line to the complex one-parameter family
$1-tz=0$.

The relation with canonical forms is therefore indirect but useful.  Harmonic
moments do not themselves define the canonical form of the original polypol.
Rather, their generating function is a Cauchy--Fantappi\`e transform of the area
measure in one complex direction.  For polygons this transform is rational and
its denominator is controlled by the vertices.  This is the point of view of
Pasechnik--Shapiro \cite{PaSh}; in their notation, the normalized harmonic
moment generating function of a signed polygonal measure supported on triangles
with vertices in a finite set $S=\{z_0,\ldots,z_n\}$ is a linear combination of
terms of the form
\[
        \frac{c_{ijk}}{(1-z_i t)(1-z_jt)(1-z_kt)}.
\]
Thus the same vertex hyperplanes which occur in the polar canonical form also
control the harmonic moment transform.  Burman--Fr\"oberg--Shapiro
\cite{BuFrSh} then study the algebraic relations between the harmonic and
anti-harmonic moments of plane polygons.

A precise comparison with the two-variable transform is obtained by putting
$z=x+iy$ and restricting to the complex line $(u,v)=(t,it)$ in the dual
$(u,v)$-plane.  If
\[
        \widehat G_\Theta(t)=\int_\Theta\frac{dx\,dy}{(1-tz)^2},
\]
then
\begin{equation}\label{eq:harmonic-F-restriction}
        \widehat G_\Theta(t)=
        \frac{1}{t^2}\int_0^t s\,F_\Theta(s,is)\,ds .
\end{equation}
Indeed, expanding both sides at the origin gives
\[
        F_\Theta(t,it)
        =2\int_\Theta\frac{dx\,dy}{(1-tz)^3}
        =\sum_{j\geq0}(j+1)(j+2)
          \left(\int_\Theta z^j\,dx\,dy\right)t^j,
\]
and integration in \eqref{eq:harmonic-F-restriction} changes the coefficient
$(j+1)(j+2)$ into $(j+1)$.  Up to the harmless convention-dependent factor
relating $dz\wedge d\bar z$ and $dx\wedge dy$, this is exactly the weighted
harmonic moment generating function $G_\Theta$.

Consequently the harmonic moment generating function is worth keeping in the
present paper, but only as a projected or one-dimensional transform of the
Fantappi\`e/canonical-form story.  It should not be advertised as a new
canonical form.  Its role is instead to record the restriction of the dual
period to the complex line $v=it$, followed by the elementary integral operator
in \eqref{eq:harmonic-F-restriction}.  In the polygonal case this restriction
is rational and reflects the polar positive geometry; in the curved polypol
case it may become algebraic or logarithmic for the same reasons as the full
transform $F_\Theta(u,v)$.

\section{Positive-geometry interpretation and further directions}
\label{sec:beyond-polytopes}

The examples and Theorem~\ref{th:dual-transform} suggest the following
interpretation.  The relation between moment-generating functions and canonical
forms is not a coincidence peculiar to polytopes.  What is special about
polytopes is that the dual period is rational and logarithmic.  Once curved
boundary components are allowed, the same incidence geometry survives but the
analytic object is generally a branched period.

Indeed, the kernel
\[
        (1-ux-vy)^{-3}
\]
introduces a moving line in the original plane.  Vertex singularities appear
when this line passes through an endpoint of a boundary arc; these are the same
hyperplanes which define the polar polygon in the polyhedral case.  Curved arcs
contribute additional singularities when the moving line becomes tangent to the
complexified boundary curve; these are precisely the projective dual curves of
the nonlinear boundary components.  Thus projective duality still controls the
singular support, even though the resulting function need not be rational.

The full disk illustrates the singularity mechanism in the cleanest possible way.  The singular locus
of
\[
        F_D(u,v)=\frac{2\pi}{(1-u^2-v^2)^{3/2}}
\]
is the dual conic to the boundary circle.  The exponent $3/2$ shows that this is
not an ordinary logarithmic canonical form, but the location of the singularity
is exactly the one predicted by the positive-geometry picture.  The upper
half-disk gives the elementary mixed example: its diameter and endpoints produce
rational vertex-type terms, while the circular arc produces the branched dual
conic term.

This leads to a useful terminology.  For a rational polypol one may regard
$F_{\cP}(u,v)du\wedge dv$ as a \emph{dual branched canonical form}, with the
understanding that this is not a canonical form in the strict sense of
Arkani-Hamed--Bai--Lam unless $\cP$ is polygonal.  The adjective ``branched''
records the replacement of logarithmic residues by monodromy of a
one-dimensional rational period.

Several problems remain.  First, one should describe the minimal
Picard--Fuchs system annihilating $F_{\cP}$ and compare its singular locus with
the vertex hyperplanes and dual boundary curves.  Secondly, the adjoint
polynomial $A_{\cP}$ in the primal canonical form should be related more
explicitly to the cancellation of spurious singularities in the boundary-period
representation of $F_{\cP}$.  Finally, it would be useful to identify natural
classes of curved polypols for which the dual period is algebraic, logarithmic,
or genuinely transcendental.

\begin{conjecture}\label{conj:singular-locus}
For a regular rational polypol $\cP$, the singular locus of the holonomic system
annihilating $F_{\cP}$ is contained in the union of the vertex hyperplanes and
the projective duals of the nonlinear boundary components.  If the boundary is
sufficiently generic, all these components occur.
\end{conjecture}

\noindent
\emph{Acknowledgements.} The author is sincerely grateful to Professor Dmitrii Pasechnik for discussions.

\end{document}